\documentclass[12pt]{amsart}
\usepackage{amssymb}

\newtheorem{theorem}{Theorem}[section]
\newtheorem{proposition}[theorem]{Proposition}

\numberwithin{equation}{section}

\begin{document}

\baselineskip=15.5pt

\title[Fourier-Mukai transform of vector bundles on surfaces]{Fourier-Mukai
transform of vector bundles on surfaces to Hilbert scheme}

\author[I. Biswas]{Indranil Biswas}

\address{School of Mathematics, Tata Institute of Fundamental
Research, Homi Bhabha Road, Mumbai 400005, India}

\email{indranil@math.tifr.res.in}

\author[D. S. Nagaraj]{D. S. Nagaraj}

\address{The Institute of Mathematical Sciences, CIT
Campus, Taramani, Chennai 600113, India}

\email{dsn@imsc.res.in}



\date{}

\begin{abstract}
Let $S$ be an irreducible smooth projective surface defined over an algebraically
closed field $k$. For a positive integer $d$, let ${\rm Hilb}^d(S)$ be the Hilbert
scheme parametrizing the zero-dimensional subschemes of $S$ of length $d$. For a
vector bundle $E$ on $S$, let ${\mathcal H}(E)\, \longrightarrow\,
{\rm Hilb}^d(S)$ be its Fourier--Mukai transform constructed using the structure
sheaf of the universal subscheme of $S\times {\rm Hilb}^d(S)$ as the kernel. We
prove that two vector bundles $E$ and $F$ on $S$ are isomorphic if the vector
bundles ${\mathcal H}(E)$ and ${\mathcal H}(F)$ are isomorphic.
\end{abstract}

\maketitle

\section{Introduction}\label{se1}

Let $S$ be an irreducible smooth projective surface defined over an algebraically
closed field. For a positive integer $d$, let $\text{Hilb}^d(S)$ denote the
Hilbert scheme that parametrizes the zero dimensional subschemes of $S$ of length
$d$. Let
$$
{\mathcal Z}\, \subset\, S\times \text{Hilb}^d(S)
$$
be the universal subscheme. Let
$$
\beta\, :\, S\times \text{Hilb}^d(S)\, \longrightarrow\, S\ ~ \text{ and } ~ \
\gamma\, :\, S\times \text{Hilb}^d(S)\,\longrightarrow\,\text{Hilb}^d(S)
$$
be the natural projections. Given a coherent sheaf $E$ on $S$, we have the
Fourier--Mukai transform
$$
{\mathcal H}(E)\,=\, \gamma_*({\mathcal O}_{\mathcal Z}\otimes \beta^*E)
\,\longrightarrow\, \text{Hilb}^d(S)\, .
$$
If $E$ is locally free, then ${\mathcal H}(E)$ is also locally free because the
restriction $$\gamma\vert_{\mathcal Z}\, :\, {\mathcal Z}\, \longrightarrow\,
\text{Hilb}^d(S)$$ is a finite and flat morphism. Therefore, this Fourier--Mukai transform
gives a map from the isomorphism classes of vector bundles on $S$ to the
isomorphism classes of vector bundles on $\text{Hilb}^d(S)$.

A natural question to ask is whether this map is injective or surjective. Note
that since
$\dim \text{Hilb}^d(S)\, >\, \dim S$ if $d\, \geq\, 2$, this map can't be
surjective when $d\, \geq\, 2$. Our aim here is to prove that this map is
injective. More precisely, we prove the following:

\begin{theorem}\label{thm0}
Two vector bundles $E$ and $F$ on $S$ are isomorphic if and only if
${\mathcal H}(E)$ and ${\mathcal H}(F)$ are isomorphic.
\end{theorem}

Theorem \ref{thm0} was proved earlier under the assumption that $S$
is a K3 or abelian surface; this was done by Addington, Markman--Mehrotra and
Meachan (see \cite{add}, \cite{yoneda}, and \cite{mea}).

\section{Vector bundles on curves and its symmetric product}

Let $k$ be an algebraically closed field. Let $C$ be an irreducible smooth projective
curve defined over $k$ of genus $g_C$, with $g_C\, \geq\, 2$. The canonical line bundle of $C$
will be denoted by $K_C$. Fix an integer $d\, \geq\, 2$. Let $S_d$ denote the group of
permutations of $\{1\, ,\cdots\, ,d\}$. The symmetric product
$$
\text{Sym}^d(C)\, :=\, C^d/S_d
$$
is the quotient for natural action of $S_d$ on $C^d$. Let
$$
{\mathcal D}\, \subset\, C\times\text{Sym}^d(C)
$$
be the universal divisor which consists of all $(x\, , \{y_1\, ,\cdots\, , y_d\})$ such that
$x\,\in\, \{y_1\, ,\cdots\, , y_d\}$. Let
\begin{equation}\label{e1}
p_1\, :\, {\mathcal D}\, \longrightarrow\, C\ ~ \text{ and } ~ \ p_2\, :\,
{\mathcal D}\,\longrightarrow\, \text{Sym}^d(C)
\end{equation}
be the projections defined by $$(x\, , \{y_1\, ,\cdots\, , y_d\})\,\longmapsto\, x~ \
\text{ and }~\ (x\, , \{y_1\, ,\cdots\, , y_d\})\,\longmapsto\, \{y_1\, ,\cdots\, ,
y_d\}$$
respectively.

For any algebraic vector bundle $E$ on $C$, define the direct image
\begin{equation}\label{e2}
{\mathcal S}(E)\, :=\, p_{2*}p^*_1E\, \longrightarrow\, \text{Sym}^d(C)\, ,
\end{equation}
where $p_1$ and $p_2$ are defined in \eqref{e1}. This ${\mathcal S}(E)$ is locally
free because $p_2$ is a finite and flat morphism.

If $0\, =\, E_0\, \subset\, E_1\, \subset\, \cdots\, \subset\, E_{m-1}\,
\subset\, E_m\,=\, E$ is the Harder--Narasimhan filtration of $E$, then define
$$
\mu_{\text{max}}(E)\, :=\, \frac{\text{degree}(E_1)}{\text{rank}(E_1)}~ \
\text{ and } ~ \  \mu_{\rm min}(E)\, :=\,
\frac{\text{degree}(E/E_{m-1})}{\text{rank}(E/E_{m-1})}\, .
$$
So $\mu_{\text{max}}(E)\,\geq\, \mu_{\rm min}(E)$, and $\mu_{\text{max}}(E)\,=
\, \mu_{\rm min}(E)$ if and only if $E$ is semistable.

\begin{proposition}\label{prop1}
Let $E$ and $F$ be vector bundles on $C$ such that
\begin{equation}\label{e}
\mu_{\rm max}(E)-\mu_{\rm min}(E)\, <\, 2(g_C-1) ~ \  \text{ and } ~ \
\mu_{\rm max}(F)-\mu_{\rm min}(F)\, <\, 2(g_C-1)\, .
\end{equation}
If the two vector bundles ${\mathcal S}(E)$ and ${\mathcal S}(F)$ (defined in \eqref{e2})
are isomorphic, then $E$ is isomorphic to $F$.
\end{proposition}

\begin{proof}
Let
$$
\varphi\, :\, C\, \longrightarrow\, \text{Sym}^d(C)
$$
be the morphism defined by $z\, \longmapsto\, d\cdot z\,=\, (z\, ,\cdots \, ,z)$. Then
$\varphi^*{\mathcal S}(E)$
admits a filtration
\begin{equation}\label{e3}
0\,=\, E(d)\, \subset\, E(d-1)\, \subset\, E(d-2)\, \subset\, \cdots\, \subset\, E(1)\,
\subset\, E(0)\,=\, \varphi^*{\mathcal S}(E)
\end{equation}
such that
\begin{equation}\label{f3}
E(d-1)\,=\, E\otimes K^{\otimes(d-1)}_C ~\ \text{ and }~\ E(i)/E(i+1)\,=\,E\otimes K^{\otimes i}_C
\end{equation}
for all $0\, \leq\, i\, \leq\, d-2$ (see \cite[p. 330, (3.7)]{BN}); in \cite{BN} it
is assumed that $k\,=\, \mathbb C$, but the proof works for any algebraically closed
field. Let
$$
0\, =\, E_0\, \subset\, E_1\, \subset\, \cdots\, \subset\, E_{m-1}\,
\subset\, E_m\,=\, E
$$
be the Harder--Narasimhan filtration of $E$. For any $j\,\in\, \mathbb Z$,
$$
\mu_{\rm max}(E\otimes K^{\otimes j}_C)\,=\, \mu_{\rm max}(E)+ 2j(g_C-1) ~ \
\text{ and }\ ~ \mu_{\rm min}(E\otimes K^{\otimes j}_C)\,=\, \mu_{\rm min}(E)+ 2j(g_C-1)\, .
$$
Hence the condition in \eqref{e} implies that
$$
\mu_{\rm max}(E\otimes K^{\otimes j}_C) \,<\, \mu_{\rm min}(E\otimes K^{\otimes (j+1)}_C)\, .
$$
Therefore, from \eqref{e3} and \eqref{f3} we conclude the following:
\begin{itemize}
\item The Harder--Narasimhan filtration of $\varphi^*{\mathcal S}(E)$ has $md$ nonzero terms.

\item If
$$
0\,=\, V_0\, \subset\, V_1\, \subset\, \cdots\,\subset\, V_{md-1}\,
\subset\, V_{md}\,=\, \varphi^*{\mathcal S}(E)
$$
is the Harder--Narasimhan filtration of $\varphi^*{\mathcal S}(E)$, then
for any $0\,\leq\, j\, \leq\, d$,
$$
V_{mj}\,=\, E(d-j)\, ,
$$
where $E(d-j)$ is the subbundle in \eqref{e3}.
\end{itemize}
More precisely, for any $0\,\leq\, j\, \leq\, d-1$ and
$0\,\leq\, i\, \leq\, m$,
$$
V_{jm+i}/V_{jm} \,=\, E_i\otimes K^{\otimes (d-j-1)}_C\, .
$$
In particular, we have
\begin{equation}\label{g1}
V_m\,=\, E(d-1) \,=\, E\otimes K^{\otimes (d-1)}_C\, .
\end{equation}

If ${\mathcal S}(E)$ and ${\mathcal S}(F)$ are isomorphic, comparing the Harder--Narasimhan
filtrations of $\varphi^*{\mathcal S}(E)$ and $\varphi^*{\mathcal S}(F)$, and using \eqref{g1},
we conclude that $E\otimes K^{\otimes (d-1)}_C$ is isomorphic to $F\otimes K^{\otimes (d-1)}_C$. This implies
that $E$ is isomorphic to $F$.
\end{proof}

In \cite[Theorem 3.2]{BN}, Proposition \ref{prop1} was proved under that assumption that both
$E$ and $F$ are semistable.

\subsection{An example}

We give an example to show that in general, ${\mathcal S}(E)\,=\,{\mathcal S}(F)$ does not
imply that $E\,=\, F$.

Note that $\text{Sym}^2(\mathbb{P}^1)\,\simeq \,\mathbb{P}^2.$
If we identify $\text{Sym}^2(\mathbb{P}^1)$ with $\mathbb{P}^2$, then the universal degree two divisor
$$
{\mathcal D}_2\, \subset\, \mathbb{P}^1\times\text{Sym}^2(\mathbb{P}^1)
\,\simeq\, \mathbb{P}^1\times \mathbb{P}^2
$$
is the zero locus of a section of the line bundle $p^*(\mathcal{O}_{\mathbb{P}^1}(2))\otimes
q^*(\mathcal{O}_{\mathbb{P}^2}(1)),$ where
\begin{equation}\label{ee1}
p\, :\, \mathbb{P}^1\times \mathbb{P}^2
\, \longrightarrow\, \mathbb{P}^1\ ~ \text{ and } ~ \ q \, :\,
\mathbb{P}^1\times \mathbb{P}^2
\,\longrightarrow\, \mathbb{P}^2
\end{equation}
are the natural projections. From  this we see that
\begin{itemize}
\item ${\mathcal S}(\mathcal{O}_{\mathbb{P}^1}(1))
\,=\, \mathcal{O}_{\mathbb{P}^2}\oplus \mathcal{O}_{\mathbb{P}^2}$

\item ${\mathcal S}(\mathcal{O}_{\mathbb{P}^1}(-1))
\,=\, \mathcal{O}_{\mathbb{P}^2}(-1)\oplus \mathcal{O}_{\mathbb{P}^2}(-1)$

\item ${\mathcal S}(\mathcal{O}_{\mathbb{P}^1})
\,=\, \mathcal{O}_{\mathbb{P}^2}\oplus \mathcal{O}_{\mathbb{P}^2}(-1)$.
\end{itemize}
For any two vector bundles $E$ and $F$ on
$\mathbb{P}^1$ we have ${\mathcal S}(E\oplus F) \,=\, {\mathcal S}(E) \oplus
{\mathcal S}(F)$. From these observations it follows that
$${\mathcal S}(\mathcal{O}^{\oplus 2}_{{\mathbb P}^1})
\,=\, \mathcal{O}_{\mathbb{P}^2}^{\oplus 2}\oplus \mathcal{O}_{\mathbb{P}^2}(-1)^{\oplus 2}
\,=\,{\mathcal S}(\mathcal{O}_{\mathbb{P}^1}(1)\oplus \mathcal{O}_{\mathbb{P}^1}(-1))\, .$$

\section{Vector bundles on surfaces and Hilbert scheme}

Let $S$ be an irreducible smooth projective surface defined over $k$. For any
$d\, \geq\, 1$, let $\text{Hilb}^d(S)$ denote the Hilbert scheme parametrizing the
$0$--dimensional subschemes of $S$ of length $d$ (see \cite{Fo}). Let
$$
{\mathcal Z}\, \subset\, S\times \text{Hilb}^d(S)
$$
be the universal subscheme which consists of all $(x\, ,z)\,\in\,
S\times \text{Hilb}^d(S)$ such that $x\,\in\, z$. Let
\begin{equation}\label{e4}
q_1\, :\, {\mathcal Z}\, \longrightarrow\, S\ ~ \text{ and } ~ \ q_2\, :\,
{\mathcal Z}\,\longrightarrow\,\text{Hilb}^d(S)
\end{equation}
be the projections defined by $(x\, ,z)\,\longmapsto\, x$ and
$(x\, ,z)\,\longmapsto\, z$ respectively.

For any algebraic vector bundle $E$ on $S$, define the direct image
\begin{equation}\label{e5}
{\mathcal H}(E)\, :=\, q_{2*}q^*_1E\, \longrightarrow\, \text{Hilb}^d(S)\, ,
\end{equation}
where $q_1$ and $q_2$ are the projections in \eqref{e4}. Since $q_2$ is a finite and flat morphism,
the direct image ${\mathcal H}(E)$ is locally free. We note that ${\mathcal H}(E)$ is the
Fourier--Mukai transform of $E$ with respect to the kernel sheaf ${\mathcal O}_{\mathcal Z}$
on $S\times \text{Hilb}^d(S)$.

\begin{theorem}\label{thm1}
Let $E$ and $F$ be vector bundles on $S$ such that ${\mathcal H}(E)$ (defined in
\eqref{e5}) is isomorphic to ${\mathcal H}(F)$. Then the two vector bundles $E$ and
$F$ are isomorphic.
\end{theorem}

\begin{proof}
If $\iota\, :\, C\, \hookrightarrow\, S$ is an embedded irreducible smooth closed curve, then
$\iota$ induces a morphism
\begin{equation}\label{g4}
\text{Sym}^d(C)\, \hookrightarrow\, \text{Hilb}^d(S)\, .
\end{equation}

Fix a very ample line bundle $\mathcal L$ on $S$. Let
\begin{equation}\label{hnf}
0\, =\, E_0\, \subset\, E_1\, \subset\, \cdots\, \subset\, E_{m-1}\,
\subset\, E_m\,=\, E
\end{equation}
be the Harder--Narasimhan filtration of $E$ with respect to $\mathcal L$. Let
$$
Y\, \subset\, S
$$
be the subset over which some $E_i$ fails to be a subbundle of $E$.
This $Y$ is a finite subset because any torsionfree
sheaf on $S$ is locally free outside a finite subset. Also note that
$Y$ is the subset over which the filtration in \eqref{hnf} fails to be
filtration of subbundles of $E$.

For $n\, \geq\, 1$, let
$$
\iota\, :\, C\, \longrightarrow\, S\, , ~\ C\, \in\, \vert {\mathcal L}^{\otimes n}\vert
$$
be an irreducible smooth closed curve lying in the complete linear system $\vert {\mathcal L}^{\otimes n}\vert$
such that $\iota(C)\bigcap Y\,=\, \emptyset$. Since $\mathcal L$ is very ample, such curves exist.

For each $1\, \leq\, i\, \leq\, m$, there is an integer $\ell_i$ such that
$\iota^* (E_i/E_{i-1})$ is semistable for a general member of $C\, \in\, \vert {\mathcal L}^{\otimes n}\vert$
if $n\, \geq\, \ell_i$ \cite[p. 221, Theorem 6.1]{MR}. Take
$$
\ell'\,=\, \text{max}\{\ell_1\, ,\cdots\, , \ell_m\}\, .
$$
If $n\, \geq\, \ell'$, then for a general $C\, \in\, \vert {\mathcal L}^{\otimes n}\vert$, the pulled back
filtration
$$
0\, =\, \iota^*E_0\, \subset\, \iota^* E_1\, \subset\, \cdots\, \subset\, \iota^*E_{m-1}\,
\subset\, \iota^*E_m\,=\, \iota^*E
$$
coincides with the Harder--Narasimhan filtration of $\iota^*E$. Indeed, this follows
immediately from the following two facts:
\begin{enumerate}
\item $\iota^* (E_i/E_{i-1})$ is semistable for a general
member of $C\, \in\, \vert {\mathcal L}^{\otimes n}\vert$ if $n\, \geq\, \ell_i$, and

\item $\mu(\iota^*(E_i/E_{i-1}))\, >\, \mu(\iota^*(E_{i+1}/E_{i}))$ because
$\mu(E_i/E_{i-1})\, >\, \mu(E_{i+1}/E_{i})$.
\end{enumerate}

Let $W$ be a vector bundle $S$. Define
$$
d_W\,:=\, c_1({\mathcal L})\cdot c_1(W)\, \in\, \mathbb Z\, .
$$
As before, let
$$
\iota\, :\, C\, \longrightarrow\, S\, , ~\ C\, \in\, \vert {\mathcal L}^{\otimes n}\vert
$$
be an irreducible smooth closed curve. We have
\begin{equation}\label{g2}
\text{degree}(\iota^* W)\, =\, n\cdot d_W\, .
\end{equation}
In other words, $\text{degree}(\iota^* W)$ depends linearly on $n$. From the
adjunction formula,
$$
2(\text{genus}(C)-1)\,=\, c_1({\mathcal L}^{\otimes n})\cdot c_1({\mathcal L}^{\otimes n}\otimes K_S)\, ,
$$
where $K_S$ is the canonical line bundle of $S$ (see \cite[p. 361, Proposition 1.5]{Ha}).
Hence we have
\begin{equation}\label{g3}
\text{genus}(C) \,=\, \frac{n^2(c_1({\mathcal L})\cdot c_1({\mathcal L}))+ nd_{K_S}+2}{2}
\end{equation}
(see \eqref{g2}). In other words, $\text{genus}(C)$ is a quadratic function of $n$.

Comparing \eqref{g2} and \eqref{g3} we conclude that there is an integer $\ell\, \geq\, \ell'$
such that for $n\, \geq\, \ell$, we have
$$
\mu(\iota^* E_1) -\mu(\iota^*(E/E_{m-1}))\, <\, 2(\text{genus}(C)-1)\, ,
$$
where $C\, \in\, \vert {\mathcal L}^{\otimes n}\vert$ is an irreducible smooth closed curve. Note that this implies
that $\text{genus}(C)\, \geq \, 2$.

Consider the embedding in \eqref{g4}. The restriction of ${\mathcal H}(E)$ (respectively,
${\mathcal H}(F)$) to $\text{Sym}^d(C)$ coincides with ${\mathcal S}(\iota^* E)$
(respectively, ${\mathcal S}(\iota^* F)$) constructed in \eqref{e2}. So
${\mathcal S}(\iota^* E)$ and ${\mathcal S}(\iota^* F)$ are isomorphic because
${\mathcal H}(E)$ and ${\mathcal H}(F)$ are isomorphic. Since
${\mathcal S}(\iota^* E)$ and ${\mathcal S}(\iota^* F)$ are isomorphic, from
Proposition \ref{prop1} it follows that $\iota^* E$ and $\iota^* F$ are isomorphic for
a general $C\, \in\, \vert {\mathcal L}^{\otimes n}\vert$ with $n\, \geq\, \ell$.

The line bundle ${\mathcal L}$ being ample, there is an integer $\ell''$
such that for every $n\, \geq\, \ell''$, we have
\begin{equation}\label{rr}
H^1(S,\, E\otimes F^*\otimes K_S\otimes {\mathcal L}^{\otimes n})\, =\, 0\, .
\end{equation}
Take $n\, \geq\, \ell''$, and let
$$
\iota\, :\, C\, \hookrightarrow\, S
$$
be any irreducible smooth closed curve lying in $\vert {\mathcal L}^{\otimes n}\vert$. Consider the short
exact sequence of sheaves
\begin{equation}\label{rr2}
0\, \longrightarrow\, F\otimes E^*\otimes {\mathcal O}_S(-C)\, \longrightarrow\,
F\otimes E^*\, \longrightarrow\, (F\otimes E^*)\vert_C \, \longrightarrow\,0\, .
\end{equation}
Since $H^1(S,\, F\otimes E^*\otimes {\mathcal O}_S(-C))\,=\,
H^1(S,\, E\otimes F^*\otimes {\mathcal L}^{\otimes n}\otimes K_S)^*$ (Serre duality), from \eqref{rr}
it follows that
$$
H^1(S,\, F\otimes E^*\otimes {\mathcal O}_S(-C))\,=\, 0\, .
$$
Therefore, from the long exact sequence of cohomology groups associated to \eqref{rr2} we conclude that
the restriction homomorphism
\begin{equation}\label{rr3}
\rho\, :\, H^0(S,\, F\otimes E^*)\, \longrightarrow\, H^0(C,\, (F\otimes E^*)\vert_C)
\end{equation}
is surjective.

Take $n\, \geq\, \text{max}\{\ell\, ,\ell''\}$, and let $C\, \in\, \vert {\mathcal L}^{\otimes n}\vert$ be a
general member. We know that $\iota^* E$ and $\iota^* F$ are isomorphic. Fix an isomorphism
$$
I\, :\, \iota^* E\,\longrightarrow\, \iota^* F\, .
$$
So $I\, \in\, H^0(C,\, \iota^* (F\otimes E^*))$. Since $\rho$ in \eqref{rr3} is surjective, there
is a homomorphism
$$
\widetilde{I}\, \in\, H^0(S,\, F\otimes E^*)
$$
such that $\rho(\widetilde{I})\,=\, I$. Let $r$ be the rank of $E$ (and also $F$). Consider the
homomorphism of line bundles
$$
\bigwedge\nolimits^r \widetilde{I}\, :\, \bigwedge\nolimits^r E\,\longrightarrow\,
\bigwedge\nolimits^r F
$$
induced by $I$. Let
$$
D(\widetilde{I})\,:=\, \text{Div}(\bigwedge\nolimits^r \widetilde{I})
$$
be the effective divisor for $\bigwedge\nolimits^r \widetilde{I}$. We know that $D(\widetilde{I})$
does not intersect $C$ because the restriction $\rho(\widetilde{I})\,=\, I$ is an isomorphism. But
$C$ is an ample effective divisor, so $C$ intersects any closed curve in $S$.
Therefore, $D(\widetilde{I})$ must be the zero divisor. Consequently, the homomorphism
$\bigwedge\nolimits^r \widetilde{I}$ is an isomorphism. This implies that $\widetilde{I}$ is an
isomorphism. So the two vector bundles $E$ and $F$ are isomorphic.
\end{proof}

\end{document}